\documentclass{article}
\usepackage{amsmath,amsthm, epsfig,amssymb, multicol, latexsym}
\usepackage{mathtools}
\usepackage[active]{srcltx}
\usepackage{tikz}
\usetikzlibrary{arrows,automata,positioning}
\usepackage[latin1]{inputenc}
\usepackage[mathscr]{euscript}
\usepackage{tabularx}
\usepackage{caption}
\newtheorem*{theorem*}{Theorem}

\usepackage{epstopdf}

\allowdisplaybreaks
\newtheorem{theorem}{Theorem}

\newtheorem{definition}{Definition}

\newtheorem{proposition}{Proposition}
\newtheorem{example}{Example}
\newtheorem{remark}{Remark}

\newcommand{\one}{{\bf 1}}
\newcommand{\dist}{{\rm dist}}
\newcommand{\D}{{\cal D}}

\title{Extreme values of the stationary distribution of random walks on directed graphs}

\author{ Sinan Aksoy \thanks{Department of Mathematics, University of California, San Diego,  La Jolla, CA 92093,
({\tt saksoy@ucsd.edu}).}
\and Fan Chung \thanks{Department of Mathematics, University of California, San Diego,  La Jolla, CA 92093,
({\tt fan@ucsd.edu}), Research is supported in part by ONR MURI N000140810747, and AFSOR AF/SUB 552082.}
\and Xing Peng  \thanks{Center for Applied Mathematics, Tianjin University,  Tianjin,    300072 China,
({\tt x2peng@tju.edu.cn}).}
}

\date{}
\begin{document}
\maketitle

\begin{abstract}
We examine the stationary distribution of random walks on directed graphs. In particular, we focus on the {\em principal ratio}, which is the ratio of maximum to minimum values of vertices in the stationary distribution. We give an upper bound for this ratio over all strongly connected graphs on $n$ vertices. We characterize all graphs achieving the upper bound and we give explicit constructions for these extremal graphs. Additionally, we show that under certain conditions, the principal ratio is tightly bounded. We also provide counterexamples to show the principal ratio cannot be tightly bounded under weaker conditions.

 \end{abstract}

\section{Introduction}

In the study of random walks on graphs, many problems that are straightforward for undirected graphs are relatively complicated in the directed case. One basic problem concerns determining stationary distributions of  random walks on simple directed graphs. For an undirected graph, the vector $\pi(v)=\frac{d_v}{\sum_v d_v}$, where $d_v$ is the degree of vertex $v$, is the unique stationary distribution if the graph is connected and non-bipartite. Consequently, the principal ratio, which is the ratio of maximum to minimum values of vertices in the stationary distribution, is $\frac{\max_v d_v}{\min_v d_v}$ and thus is at most $n$, the number of vertices.

In contrast, the directed case is far more subtle: not only does no such closed form solution exist for the stationary distribution, but its principal ratio can be exponentially large in $n$. This has immediate implications for the central question of bounding the rate of convergence of a random walk on a directed graph where extreme values of the stationary distribution play an important role in addition to eigenvalues. For example, it can be shown that for a strongly connected directed graph, the order of the rate of convergence is bounded above by $2 \lambda_1^{-1} (-\log (\min_x \pi(x)))$, where $\lambda_1$ is the first nontrivial eigenvalue of the normalized Laplacian of the directed graph, as defined in \cite{chung2}. Namely after at most
$t \geq 2 \lambda_1^{-1} (-\log (\min_x \pi(x))+2c)$ steps, the total variation distance is at most $e^{-c}$.


 Another application of the stationary distribution and its principal ratio is in the algorithmic design and analysis of vertex ranking, for so-called ``{\it PageRank}" algorithms for directed graphs (since many real-world information networks are indeed directed graphs). PageRank algorithms \cite{acl}  use a variation of random walks with an additional diffusion parameter and therefore it is not surprising that the effectiveness of the algorithm depends on the principal ratio.

In addition to its role in Page Rank algorithmic analysis and bounding the rate of converge in random walks, it has been noted (see \cite{cg}) that the principal ratio can be interpreted as a numerical metric for graph irregularity since it achieves its minimum of $1$ only for regular graphs.


The study of the principal ratio of the stationary distribution has a rich history. We note that the stationary distribution is a special case of the Perron vector, which is the unique positive eigenvector associated with the largest eigenvalue of an irreducible matrix with non-negative entries. There is a large literature examining the Perron vector of the adjacency matrix of undirected graphs, which has been studied by  Cioab\u{a} and Gregory \cite{cg}, Tait and Tobin \cite{taitTobin}, Papendieck and Recht \cite{pr}, Zhao and Hong \cite{zh}, and Zhang \cite{zhang}. In this paper, we focus on principal ratio of the stationary distribution of random walk on a strongly connected directed graph with $n$ vertices.

For directed graphs, some relevant prior results are from matrix analysis. Latham \cite{la}, Minc \cite{minc}, and Ostrowski \cite{os1} studied the Perron vector of a (not necessarily symmetric) matrix with positive entries, which can be used to study matrices associated with complete, weighted directed graphs. However, for our case a relevant prior result comes from Lynn and Timlake, who gave bounds of the principal ratio for {\em primitive} matrices with non-negative entries (see Corollary 2.1.1 in \cite{LT}).  As we will soon further explain, since ergodic random walks on directed graphs have primitive transition probability matrices, their result applies naturally in our setting. Letting $\gamma(D)$ denote the principal ratio of a directed graph $D$, their result yields the bound
\[
\gamma(D)\leq (1+o(1))(n-1)^{n-1},
\]
where $D$ is a strongly connected, aperiodic directed graph on $n$ vertices.

Chung gave an upper bound (see \cite{chung2})  on the principal ratio of a strongly connected directed graph $D$ that depends on certain graph parameters. Namely,
\[
\gamma(D)\leq k^d,
\]
where $d$ is the diameter of the graph and $k$ is the maximum out-degree. Since $d,k\leq n-1$, this bound also implies absolute upper bound on the principal ratio of $(n-1)^{n-1}$ over all strongly connected directed graphs on $n$ vertices.


In this paper, we provide an exact expression for the maximum of the principal ratio over all strongly connected directed graphs on $n$ vertices. Asymptotically, our bound is
\begin{align*}
\gamma(n)=\max_{D: |V(D)|=n}\gamma(D)= \left(\frac{2}{3} + o(1)\right)(n-1)!.
\end{align*}

Furthermore, we show that this bound is achieved by precisely three directed graphs, up to isomorphism.

In addition to an extremal analysis of the principal ratio, we also examine conditions under which the principal ratio can be tightly bounded. Namely, we show that if a directed graph satisfies a degree condition and a discrepancy condition, then its principal ratio can be tightly bounded in the sense that it is ``close" to the minimum possible value of 1. Furthermore, we provide counterexamples that show the principal ratio cannot be tightly bounded if either the discrepancy condition or degree conditions are removed.

\section{Random walks on directed graphs}

Let $D$ be a directed graph with vertex set $V(D)$ and edge set $E(D)$. A directed edge from vertex $u$ to $v$ is denoted by $(u,v)$ or $u \to v$, and we say $v$ is an out-neighbor of $u$ or $u$ is an in-neighbor of $v$. We assume $D$ is simple, meaning $D$ has no loops or multiple edges. For each $u \in V(D)$, the {\it out-neighborhood of $u$}, denoted by $N_D^{+}(u)$, is the vertex set $\{v: (u,v) \in E(D)\}$ and the {\it out-degree of $u$}, denoted by $d_D^+(u)$, is $|N_D^{+}(u)|$.  Similarly, the in-neighborhood and in-degree of $u$ are denoted by $N_D^{-} (u)$ and $d_D^-(u)$ respectively. We will omit the subscript $D$ whenever $D$ is clear from context. A {\it walk} is a sequence of vertices $(v_0,v_1,\dots,v_k)$ where $(v_i,v_{i+1})$ is an edge.

A random walk on a directed graph is defined by a transition probability matrix $P$, where $P(u,v)$ denotes the probability of moving from vertex $u$ to vertex $v$. In this paper, we consider {\em simple} random walks in which moving from a vertex to any of its neighbors is equally likely. Accordingly, the probability transition matrix $P$ is defined by
\[
P(u,v) =
\begin{cases}
\frac{1}{d^+(u)}, & \text{if } (u,v) \ \text{is an edge}, \\
0 & \text{otherwise}.
\end{cases}
\]

While we assume $P$ is of the above form and consider only random walks on directed graphs with unweighted edges, we note that every finite Markov chain can be viewed as a random walk on a weighted directed graph. Namely, if $w_{uv}\geq 0$ denote edge weights, a general probability transition matrix $P$ can be defined as

\[
P(u,v)=\frac{w_{uv}}{\sum_{z}w_{uz}}.
\]

A probability distribution is a function $\pi: V(G) \rightarrow \mathbb{R}^+ \cup \{0\}$ satisfying $\sum_v \pi(v)=1$ and is said to be a {\em stationary distribution} of a random walk if

\[
\pi P=\pi,
\]

\noindent where $\pi$ is viewed as a row vector. It can be easily shown that $\pi(v)=\frac{d_v}{\sum_u d_u}$ is a stationary distribution for a simple random walk on {\em any} undirected graph and is unique if the graph is connected.  For a strongly connected directed graph,    the existence of a stationary distribution is guaranteed by the celebrated Perron-Frobenius Theorem.
Since  $\sum_v P(u,v)=1$ for strongly connected directed graphs,
\[
P\one=\one,
\]

\noindent and thus the all ones vector $\one$ is trivially the right Perron eigenvector associated with eigenvalue 1. By the Perron Frobenius theorem, there exists a left (row) eigenvector $\phi$ with positive entries such that

\[
\phi P =\phi.
\]

We may scale $\phi$ so that $\sum_u \phi(u)=1$, in which case $\phi$ is the (unique) stationary distribution which we refer to as the Perron vector. While there is no closed formula for a stationary distribution of a strongly connected directed graph in general, a closed formula does exist for those in which the in-degree of each vertex is equal to its out-degree.

\begin{example}  Eulerian directed graphs have stationary distribution proportional to their out-degree sequences, $\phi(v)=\frac{d_v^+}{\sum_{u}d_u^+}$. Consequently, the stationary distribution of a directed regular graph with in-degrees and out-degrees all equal is given by the uniform distribution, $\phi=\one/n$.
\end{example}

The {\it principal ratio}  $\gamma(D)$ of a strongly connected, directed graph $D$ is denoted by
\[
\gamma(D)=\frac{\max_u \phi(u)}{\min_u \phi(u)}.
\]

For Eulerian directed graphs, the principal ratio is the ratio of the largest to smallest out-degree. Since regular directed graphs are Eulerian with out-degrees all equal, they achieve the minimum possible principal ratio of 1. Thus the principal ratio can be regarded as one numerical measure of a directed graphs irregularity.

A random walk is ergodic if for any initial distribution $f$, the random walk converges to the unique stationary distribution, i.e.,

\[
\lim_{k \to \infty} fP^k=\phi.
\]

For undirected graphs, the spectral decomposition of $P$ shows a random walk is ergodic if and only if the graph is connected and non-bipartite. However, the directed case requires a more nuanced criterion. For example, while a random walk on an undirected cycle $C_n$ with $n$ odd is ergodic, a random walk on a directed cycle $C_n$ is not. A random walk on a directed graph is ergodic if and only if $D$ is strongly connected and {\em aperiodic}, i.e., the greatest common divisor of the lengths of all its directed cycles is 1.

Directed graphs which are both strongly connected and aperiodic have {\it primitive} transition matrices. That is, for such graphs there exists some integer $k$ such that all entries of $P^k$ are positive.






\section{A sharp upper bound on the principal ratio}

In this paper, we will prove an upper bound on the principal ratio in terms of $n$ that is best possible. For $n\geq 3$, we define a function

\[
\gamma(n)=\max\{\gamma(D): D \text{ is strongly connected  with } n \textrm{ vertices}\}.
\]

\noindent We will show:

\begin{theorem} \label{asymThm}
The maximum of the principal ratio of the stationary distribution over all strongly connected directed graphs on $n$ vertices is asymptotically
\begin{align*}
\gamma(n)&=\left( \frac{2}{3}+o(1) \right)(n-1)!.
\end{align*}
\end{theorem}

This theorem is an immediate consequence of the following theorem which we prove.

\begin{theorem} \label{thm1}
The maximum of the principal ratio of the stationary distribution over all strongly connected directed graphs on $n\geq 3$ vertices is exactly
\begin{align*}
\gamma(n) &= \frac{2}{3}\left(\frac{n}{n-1}+\frac{1}{(n-1)!}\sum_{i=1}^{n-3}i!\right)(n-1)!.
\end{align*}

Moreover, $\gamma(n)$ is attained only by directed graphs $D_1, D_2,$ and $D_3$ defined as follows: $D_1, D_2$, and $D_3$ have vertex set  $\{v_1,v_2,\ldots,v_n\}$ and edge set
\[
E(D)=\{(v_i,v_{i+1}): \textrm{ for all } 1 \leq i \leq n-1)\} \cup  \{(v_j,v_i): \textrm{ for all } 1 \leq  i < j \leq n-1\} \cup S(D),
\]
where
\[
S(D) =
\begin{cases}
\{(v_n,v_1)\} & \text{for } D=D_1  \\
\{(v_n,v_2) \} & \text{for } D=D_2 \\
\{(v_n,v_1),(v_n,v_2)\} & \text{for } D=D_3\\
\end{cases}.
\]

The case for $n=5$ is illustrated in Figure $\ref{3const}$.
\end{theorem}

\begin{figure}[t]
\centering
\includegraphics[scale=.5]{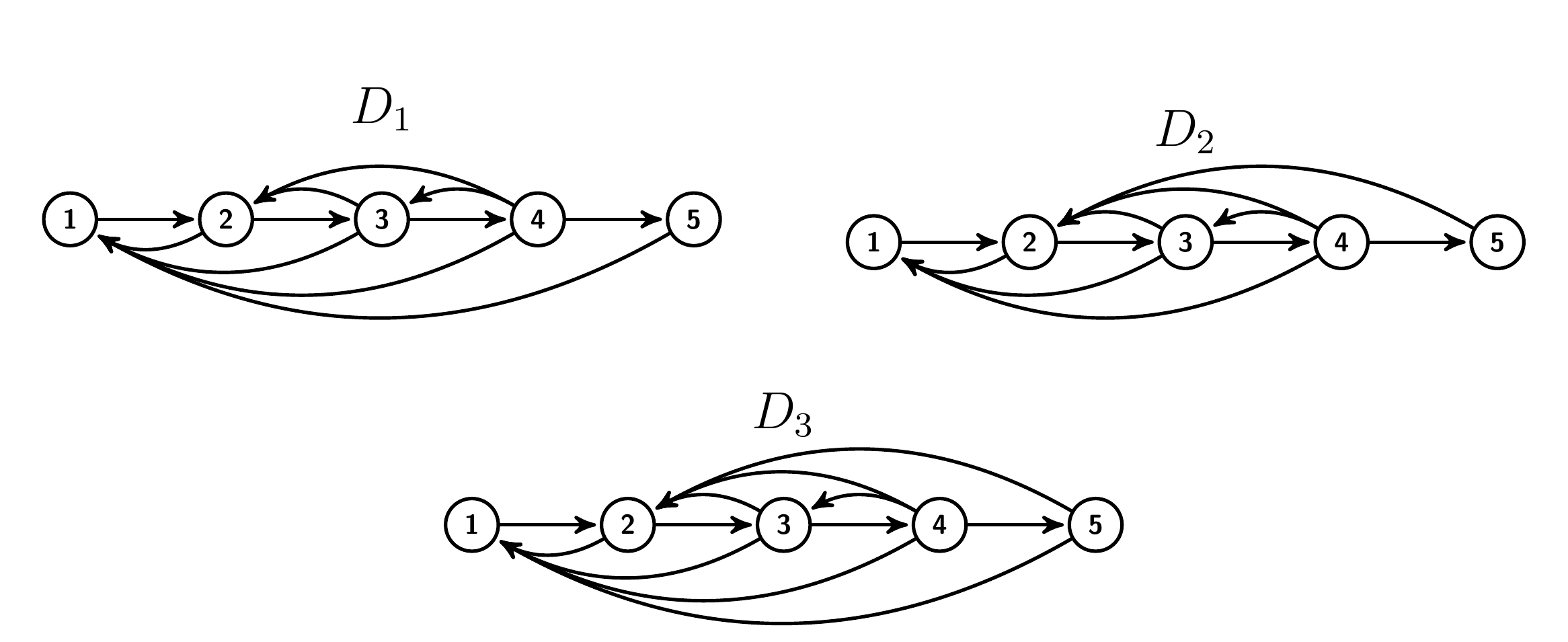}
\caption{The three constructions $D_1,D_2,D_3$ for $n=5$.} \label{3const}
\end{figure}

\begin{remark}
Note that the extremal graphs $D_1,D_2,D_3$ are not only strongly connected, but also aperiodic. Thus, Theorem 1 still holds if one restricts attention to stationary distributions of ergodic random walks.
\end{remark}

The proof of Theorem 2 follows from a sequence of propositions. The basic idea is as follows: we first show that if the principal ratio of a directed graph achieves the bound in Theorem \ref{thm1}, then the graph must necessarily satisfy a set of properties, which are described in Section $\ref{extremalStructure}$. In Sections $\ref{sec:addEdge}-\ref{sec:deleteEdge}$, we identify families of graphs that satisfy these properties, but nonetheless are not extremal. Namely, given an arbitrary member from this family, we describe how one can modify this graph by adding or deleting edges so that its principal ratio strictly increases. In Section $\ref{sec:mainThmProof}$, we apply these propositions to show that unless a given graph is one of  three graphs, it can be modified to increase its principal ratio. Finally, after establishing that all three of these extremal graphs indeed have the same principal ratio, we finish the proof and we explicitly compute the stationary distribution of one of these extremal graphs.

\begin{remark}
Note that the graphs $D_1$ and $D_2$ are proper subgraphs of $D_3$. While all three graphs have different stationary distributions, their principal ratios are nonetheless equal.
\end{remark}

\section{The structure of the extremal graphs}\label{extremalStructure}

We assume all directed graphs $D$ are strongly connected. For two vertices $u$ and $v$, the distance $\dist(u,v)$ is the number of edges in a shortest directed path from $u$ to $v$. For two subsets $V_1,V_2$,  the directed distance $\dist(V_1,V_2)$ from $V_1$ to $V_2$ is defined as $\min\{\dist(u,v): u \in V_1 \textrm{ and } v \in V_2\}$.  For a directed graph $D$, let $\phi$ be the (left) eigenvector corresponding to the eigenvalue 1 for the transition probability matrix $P$. We define two subsets of $V(D)$ with respect to $\phi$ as follows.
\[
V_{\max}=\{v \in V(D) : \max_{u \in V(D)} \phi(u)=\phi(v)\}.
\]
\[
V_{\min}=\{v \in V(D) : \min_{u \in V(D)} \phi(u)=\phi(v)\}.
\]

We will establish a number of useful facts that relate the ratio of values of vertices of the Perron vector to the distance between those vertices.

\begin{proposition} \label{keylemma1}
If $v_1,v_2,\ldots,v_k$ is a path of length $k-1$ from $v_1$ to $v_k$, then
\[
\frac{\phi(v_1)}{\phi(v_k)} \leq \prod_{i=1}^{k-1} d^+(v_i).
\]
\end{proposition}

\begin{proof}
From $\phi P^k=\phi$, we obtain
\[
\phi(v_k)=\sum_{z \in V(D)} \phi(z)P^k(z,v_k) \geq  \phi(v_1)P^k({v_1,v_k}).
\]
By considering the  path $v_1,v_2,\ldots,v_k$, we have
\[
P^k(v_1,v_k) \geq \prod_{i=1}^{k-1} \frac{1}{d^+(v_i)}.
\]
Equivalently,   $\tfrac{\phi(v_1)}{\phi(v_k)} \leq \prod_{i=1}^{k-1} d^+(v_i)$.

\end{proof}
\begin{proposition} \label{keylemma2}
If $\mbox{\emph{dist}}(u,v) = k$, then
\[
 \frac{\phi(u)}{\phi(v)} \leq (n-1)_k,
\]
where $(n-1)_k= (n-1) \cdot (n-2) \cdots (n-k)$ is the falling factorial.
\end{proposition}

\begin{proof}
Let $\mathcal P=\{u=v_0,v_1,\ldots,v_k=v\}$ be a shortest path from $u$ to $v$.
 For all $0 \leq i \leq k-2$ and $ j \geq i+2$, we note that $(v_i,v_j)$ is {\em not} a directed edge. Since $D$ has no loops, we have $d^+(v_i) \leq n-k+i$ for all $0 \leq i \leq k-1$. The proposition now follows by applying Proposition \ref{keylemma1}.
\end{proof}

\begin{proposition}\label{reducedist}
For any directed graph $D$ with $n$ vertices, we have $\dist(V_{\max},V_{\min}) \leq n-2$.
\end{proposition}

\begin{proof}
Suppose  $\dist(V_{\max},V_{\min}) =n-1=\dist(u,v)$  for some  $u \in V_{\max}$ and $v \in V_{\min}$. Let ${\cal P}=v_1,v_2,\ldots,v_n$ be a shortest directed path of length $n-1$ such that $v_1=u$ and $v_n=v$. Since $\cal P$ is a shortest directed path, we note $v_2$ is the only outneighbor of $v_1$.  From $\phi P=\phi$, we obtain
 \[
 \phi(v_2)=\phi(v_1)+\sum_{\underset{v_j \rightarrow v_2}{j \geq 3}} \frac{\phi(v_j)}{d^+(v_j)}.
 \]
Thus $\phi(v_2) \geq \phi(v_1)$ and so $\dist(V_{\max},V_{\min}) \leq \dist(v_2,v_n) \leq n-2$, which is a contradiction.
\end{proof}


\begin{proposition}\label{smalldist}
For a directed graph $D$ with $n$ vertices,   if $\dist(V_{\max},V_{\min}) \leq n-3$, then $\gamma(D) \leq \tfrac{1}{2}(n-1)!$.
\end{proposition}

\begin{proof}
Let $u \in V_{\max}$ and $v \in V_{\min}$ such that $\dist(u,v)=\dist(V_{\max},V_{\min})$.  By Proposition \ref{keylemma2}, we have $\gamma(D) \leq (n-1)_{n-3}=\tfrac{1}{2}(n-1)!$.
\end{proof}

\begin{proposition} \label{degree}
Let $D$ be a strongly connected  directed graph with vertex set $\{v_1,\ldots,v_n\}$. Assume $v_1,v_2,\ldots,v_n$ is a shortest directed path from $v_1$ to $v_n$. Suppose $v_2 \in V_{\max}$ and $v_n \in V_{\min}$.  If $\gamma(D) > \tfrac{2}{3}(n-1)!$, then we have $N^+(v_2)=\{v_1,v_3\}$, $N^+(v_3)=\{v_1,v_2,v_4\}$, and $d^+(v_i) \geq \lfloor \tfrac{2i}{3} \rfloor$ for $4 \leq i \leq n-1$.
\end{proposition}

\begin{proof}

Since $v_1,\dots, v_n$ is a shortest path from $v_1$ to $v_n$, we have $d^+(v_i)\leq i$. To prove $N^+(v_2)=\{v_1,v_3 \}$ and $N^+(v_3)=\{v_1,v_2,v_4 \}$, it therefore suffices to show $d^+(v_2)=2$ and $d^+(v_3)=3$. From $\phi P=\phi$, we have for $1\leq j \leq n-1$,

\begin{align*}
\phi(v_{j+1})=\frac{\phi(v_j)}{d^+(v_j)}+ \sum_{\substack{i \geq j+2 \\ v_i \to v_j}} \frac{\phi(v_i)}{d^+(v_i)}&\geq \frac{\phi(v_j)}{d^+(v_j)} \geq \frac{\phi(v_j)}{j}.
\end{align*}

If $d^+(v_2)=1$, then applying the above bound we have $\phi(v_n)\geq\tfrac{\phi(v_2)}{(n-1)\dots 4\cdot3}$, yielding the contradiction $\gamma(D)\leq \frac{1}{2}(n-1)!$. Similarly, if $d^+(v_3)\leq 2$, or if $d^+(v_i) < \lfloor \tfrac{2i}{3} \rfloor$ for some $i$ where $4 \leq i \leq n-1$, then applying the above bound yields $\gamma(D)\leq \frac{2}{3}(n-1)!$.

\end{proof}

\begin{proposition} \label{special}
Let $D$ be a strongly connected  directed graph with vertex set $\{v_1,v_2,\ldots,v_n\}$. Assume $v_2,\ldots,v_n$ is a shortest directed path from $v_2$ to $v_n$, where  $v_2 \in V_{\max}$ and $v_n \in V_{\min}$ such that $\dist(V_{\max},V_{\min})=n-2$. If $\gamma(D) > \tfrac{2}{3}(n-1)!$, then we have $(v_1,v_2),(v_2,v_1) \in E(D)$ and $v_2$ is the only out-neighbor of $v_1$.
\end{proposition}
\begin{proof}
We first show $(v_2,v_1)$ must be an edge. Suppose not. Then $v_3$ will be the only outneighbor of $v_2$. The equation $\phi P=\phi$  gives
\[
\phi(v_3)=\phi(v_2)+\sum_{\underset{v_j \to v_3}{j \geq 4}} \frac{\phi(v_j)}{d^+(v_j)}.
\]
Therefore, $\phi(v_3) \geq \phi(v_2)$ which yields that $v_3 \in V_{\max}$ and $\dist(V_{\max},V_{\min}) \leq n-3$. By Proposition \ref{smalldist}, we have $\gamma(D) \leq \tfrac{1}{2}(n-1)!$ which is a contradiction.  Therefore, $(v_2,v_1)$ is an edge.

Next, we will show $N^+(v_1)=\{v_2 \}$. Since we assume $v_2,\ldots,v_n$ is a shortest path from $v_2$ to $v_n$, we have $N^+(v_2)=\{v_1,v_3\}$ and $N^+(v_1) \subseteq \{v_2,v_3,v_4 \}$. Moreover, we have $d^+(v_i) \leq i$ for $3 \leq i \leq n$ as $N^+(v_i) \subseteq \{v_1,\ldots,v_{i-1}\} \cup \{v_{i+1}\}$. Lastly, we note that from $\phi P= \phi$, we have $\phi(v_1) \geq \frac{1}{2} \phi(v_2)$. Assume $v_4 \in N^+(v_1)$. Then by considering directed paths $v_1,v_4,\dots,v_n$ and $v_2,v_3,\dots,v_n$ and applying Proposition \ref{keylemma1}, we have

\begin{align*}
\phi(v_n) & \geq \frac{\phi(v_2)}{d^+(v_2)\dots d^+(v_{n-1})} + \frac{\phi(v_1)}{d^+(v_1)\cdot d^+(v_4)\dots d^+(v_{n-1})} \geq \frac{\phi(v_2)}{(n-1)!} + \frac{\phi(v_1)}{(n-1)_{n-3}} \geq \frac{2 \phi(v_2)}{(n-1)!},
\end{align*}

\noindent yielding the contradiction $\gamma(D)=\frac{\phi(v_2)}{\phi(v_n)}\leq\frac{1}{2}(n-1)!$. So, $N^+(v_1) \subseteq \{v_2,v_3 \}$. Assume $v_3 \in N^+(v_1)$. Again, by considering directed paths $v_1,v_3,\dots,v_n$ and $v_2,v_3,\dots,v_n$ and applying Proposition \ref{keylemma1}, we similarly obtain

\begin{align*}
\phi(v_n) & \geq \frac{\phi(v_2)}{d^+(v_2)\dots d^+(v_{n-1})} + \frac{\phi(v_1)}{d^+(v_1)\cdot d^+(v_3)\dots d^+(v_{n-1})} \geq \frac{\phi(v_2)}{(n-1)!} + \frac{\phi(v_1)}{(n-1)_{n-2}} \geq \frac{3 \phi(v_2)}{2(n-1)!},
\end{align*}

\noindent yielding the contradiction $\gamma(D)=\frac{\phi(v_2)}{\phi(v_n)}\leq \frac{2}{3}(n-1)!$. Thus $v_3 \not \in N^+(v_1)$ and since $D$ is strongly connected, $N^+(v_1)\not= \varnothing$. Therefore, $N^+(v_1)=\{v_2 \}$.

\end{proof}

\section{Adding edges to increase the principal ratio}\label{sec:addEdge}

Based on Propositions 1-6, we consider the definition of the following family of graphs. An extremal graph must satisfy (i)-(iv) in the definition below.



\begin{definition} \label{def:Dn}
For each  $n$,  let $\D_n$ be a family of directed graphs where each $D \in \D_n$ on vertex set $\{v_1,\dots,v_n\}$ satisfies the following properties:
\begin{description}
\item [(i)] The shortest path from $v_1$ to $v_n$ is of length $n-1$ and is denoted by $v_1,v_2,\ldots,v_n$.
\item[(ii)] For $i \in \{2,3\},  d^+(v_i)=i$.
\item[(iii)] For each $4 \leq i \leq n-1 $, we have $d^+(v_i) \geq  \lfloor \tfrac{2i}{3} \rfloor$.
\item[(iv)] $v_2 \in V_{\max}$, $v_n \in V_{\min}$, and  $\dist( V_{\max}, V_{\min})=\dist(v_2,v_n)=n-2$.
\item[(v)] There exist $i$ and $j$ such that $(v_j,v_i)$ is not an edge where $4 \leq j \leq n-1$ and $1 \leq i \leq j-1$.
\end{description}
\end{definition}

For each $D\in \D_n$, we now define an associated graph $D^+$ identical to $D$ except for the addition of a single edge.

\begin{definition} \label{DPlus}
For a given $D \in \D_n$, let $4 \leq t \leq n$ denote the smallest integer and  $s<t$ the largest integer such that $(v_t,v_s)$ is not an edge of $D$. Define $D^+$  as the directed graph with the same vertex set as $D$ and with edge set $E(D) \cup \{(v_t,v_s)\}$, as illustrated in Figure $\ref{fig:DPlus}$.
\end{definition}

\begin{figure}[t]
\centering
\includegraphics[scale=.5]{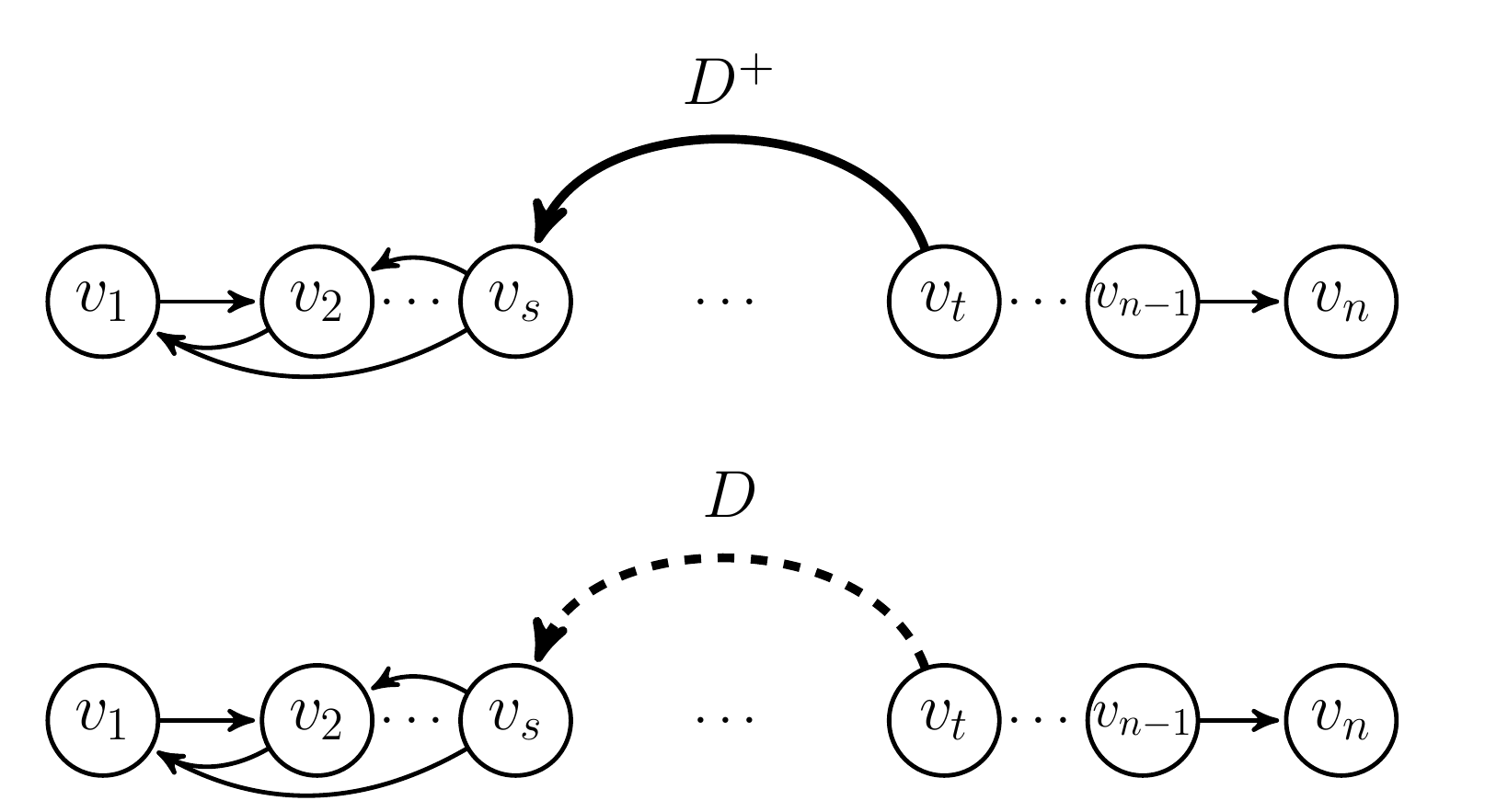}
\caption{$D$ and $D^+$. A dashed edge indicates the absence of that edge.} \label{fig:DPlus}
\end{figure}

For a given $D \in \D_n$, we wish to compare the principal ratios of $D$ and $D^+$. In order to do so, must establish some tools used to compare their stationary distributions. First, the following proposition provides a useful way to express entries of the Perron vector as a multiple of a single entry.

\begin{proposition} \label{writeFG}
Let $D$ be a directed graph whose vertex set is $\{v_1,\ldots,v_n\}$. We assume $v_1,\ldots,v_n$ is a shortest path from $v_1$ to $v_n$.  If $\phi$ is the Perron vector of the transition probability matrix $P$, then for $1\leq i \leq n$,  there exists a function $f_i$ such that
\begin{align*}
\phi(v_i) &= f_i\cdot \phi(v_n),
\end{align*}
where the the functions $f_i$ satisfy

\begin{equation} \label{recursion}
f_k =\frac{f_{k-1}}{d^+(v_{k-1})}+\sum_{\substack{i \geq k+1 \\ v_i \to v_k}} \frac{f_i}{d^+(v_i)}.
\end{equation}

\end{proposition}
\begin{proof}
We proceed by induction. Trivially, $f_n=1$. Let $1<k<n-1$. Assume the proposition holds for all integers $j$ where $k \leq j \leq n$. We show the result holds for $i=k-1$. As $\phi=\phi P$, we have
\[
\phi(v_k) =\frac{\phi(v_{k-1})}{d^+(v_{k-1})}+\sum_{\substack{i \geq k+1 \\ v_i \to v_k}} \frac{\phi(v_i)}{d^+(v_i)}.
\]
We note if $k=n$, then we do not have the second term of the equation above. Applying the induction hypothesis and rearranging the above yields
\begin{align}
f_{k-1} &=d^+(v_{k-1})\left(f_k-\sum_{\substack{i \geq k+1 \\ v_i \to v_k}} \frac{f_i}{d^+(v_i)} \right).
\end{align}
\end{proof}

The upshot of Proposition $\ref{writeFG}$ is that when comparing two graphs $D$ and $D'$ where $V(D)=V(D')=\{v_1,\ldots,v_n\}$ and $v_1,\ldots,v_n$ is a shortest directed path from $v_1$ to $v_n$ in $D$ and $D'$, we may write their Perron vectors entrywise as
\begin{align*}
\phi(v_i)&=f_i \cdot \phi(v_n)\\
\psi(v_i)&=g_i \cdot \psi(v_n),
\end{align*}
for some functions $f_i$ and $g_i$ satisfying (\ref{recursion}). The following proposition describes when $f_i=g_i$.

\begin{proposition} \label{equalpart}
Let $D$ and $D'$ and their respective Perron vectors be as described above.


If there is some $1 \leq s \leq n-1$ such that  $d_{D}^+(v_i)=d_{D'}^+(v_i)$ for each $s \leq i \leq n$, then we have $f_i=g_i$ for each $s \leq i \leq n$.
\end{proposition}
This proposition can be proved  inductively  by using \eqref{recursion} and we skip the proof here. The next proposition compares $f_i$ and $g_i$ for the graphs $D$ and $D^+$.


\begin{proposition} \label{addedge0}
For each  $D \in \D_n$, let $D^+$ be as defined in Definition $\ref{DPlus}$. Suppose $\phi$ and $\psi$ are the Perron vectors of the transition probability matrices of $D$ and $D^+$ respectively. Moreover, suppose $\phi(v_i)=f_i \cdot \phi(v_n)$ and  $\psi(v_i)=g_i \cdot \psi(v_n)$ for each $1 \leq i \leq n$. We have
\begin{enumerate}
\item[(a)] $f_i=g_i$ for each $t+1 \leq i \leq n$.
\item[(b)] $\tfrac{g_t}{f_t}=\frac{d^+_D(v_t)+1}{d^+_D(v_t)}$.
\item[(c)] $\tfrac{g_{t-1}-f_{t-1}}{t-1}=\tfrac{g_t}{d_D^+(v_t)+1}=\tfrac{g_{t-2}-f_{t-2}}{(t-1)_2}.$
\end{enumerate}
If $t \geq 5$, then additionally we have
\begin{enumerate}
\item[(d)] For each $3 \leq k \leq t-2$, we have $ \tfrac{g_{t-k}-f_{t-k}}{(t-1)_k}  \geq \tfrac{g_{t}}{d_D^+(v_t)+1} \left( 1-\tfrac{4}{3} \sum_{j=1}^{k-2} \tfrac{1}{(t-j)_2} \right) > 0$.
\item[(e)] For each $3 \leq k \leq t-2$, we have $\tfrac{g_{t-k}-f_{t-k}}{(t-1)_k} \leq \tfrac{g_{t}}{d_D^+(v_t)+1}$.
\end{enumerate}
\end{proposition}
\begin{proof}
Since $4 \leq s \leq n-1$ is the smallest integer such that  an edge $(v_t,v_s)$ is missing for some $1 \leq s \leq t-1$, we have $d^+(v_i)=i$ for each $2 \leq i \leq t-1$. We also note $d_{D}^+(v_i)=d_{D^+}^+(v_i)$ for each $1 \leq i \not =t \leq n$ and $d_{D}^+(v_t)+1=d_{D^+}^+(v_t)$.

 Part (a) follows from Proposition \ref{equalpart} easily.  Part (b) can be verified  by using the equation \eqref{recursion}.
 If $t \in \{3,4\},$ then we do not need Part (d) or Part (e).  We can compute Part (c) directly by using the out-degree conditions and the equation \eqref{recursion}.  
 
 For Part (d) and Part Part (e), we first  prove them simultaneously by induction  on $k$ for $3 \leq k \leq t-s-1$.  We mention here for the case where $k=t-s$, we will give the argument separately. If either $t=s+1$ or $t=s+2$, then we prove directly for $k=t-s$ and for $t-s+1 \leq k \leq t-2$ the proof is by induction.

 The base case is $k=3$.  From \eqref{recursion}, we have
\[
g_{t-2}=\frac{g_{t-3}}{t-3}+\frac{g_{t-1}}{t-1}+\sum_{\underset{v_j \to v_{t-2}}{ j \geq t}} \frac{g_j}{d_{D^+}^+(v_j)}
\]
\[
f_{t-2}=\frac{f_{t-3}}{t-3}+\frac{f_{t-1}}{t-1}+\sum_{\underset{v_j \to v_{t-2}}{ j \geq t}} \frac{f_j}{d^+(v_j)}
\]
We note $\tfrac{f_j}{d_D^+(v_j)}=\tfrac{g_j}{d_{D^+}^+(v_j)} $ for all $j \geq t+1$.  Combining with Part (b), we have
\[
g_{t-2}-f_{t-2}=\frac{g_{t-3}-f_{t-3}}{t-3}+\frac{g_{t-1}-f_{t-1}}{t-1}.
\]
We solve for $\tfrac{g_{t-3}-f_{t-3}}{t-3}$ and divide  both sides of the resulted equation by $(t-1)_2$.  Then Part (c) gives the base case of Part (d) and Part (e).

For the inductive step, we assume Part (d) and Part (e) hold for all $3 \leq j \leq k-1$.  As  for the base case,  from equation \eqref{recursion},   $g_{t-k}$ satisfies the following equation:
\[
g_{t-k}=\frac{g_{t-k-1}}{t-k-1}+\sum_{1 \leq j \leq k-1} \frac{g_{t-j}}{t-j}+\sum_{\underset{v_j \to v_{t-k}}{ j \geq t}} \frac{g_j}{d_{D^+}^+(v_j)}.
\]
Similarly,
\[
f_{t-k}=\frac{f_{t-k-1}}{t-k-1}+\sum_{1 \leq j \leq k-1} \frac{f_{t-j}}{t-j}+\sum_{\underset{v_j \to v_{t-k}}{ j \geq t}} \frac{f_j}{d_{D}^+(v_j)}.
\]
Solving for $\tfrac{g_{t-k-1}-f_{t-k-1}}{t-k-1}$ and dividing both sides of the equation by $(t-1)_k$, we have
\[
\frac{g_{t-k-1}-f_{t-k-1}}{(t-1)_{k+1}}=\frac{g_{t-k}-f_{t-k}}{(t-1)_k}-\sum_{j=1}^{k-1} \frac{g_{t-j}-f_{t-j}}{(t-j)(t-1)_k}.
\]
We note $g_{t-j}-f_{t-j}>0$ for each $1 \leq j \leq k-1$ by the inductive hypothesis of Part (d). 
 Part (e) then follows  from the  inductive hypothesis of Part (e).
 
 Applying Part (c) as well as the inductive hypothesis for Part (e), we have

\[
\frac{g_{t-k-1}-f_{t-k-1}}{(t-1)_{k+1}} \geq  \frac{g_t}{d_D^+(v_t)+1} \left( 1-\frac{4}{3} \sum_{j=1}^{k-2} \frac{1}{(t-j)_2} - \sum_{j=1}^{k-1} \frac{1}{(t-j)_{k-j+1}}  \right),
\]
since
\begin{align*}
\sum_{j=1}^{k-1} \frac{1}{(t-j)_{k-j+1}}&=\frac{1}{(t-k+1)_2} \left(1+\frac{1}{t-k+2}+\frac{1}{(t-k+3)_2}+\cdots+\frac{1}{(t-1)_{k-2}} \right) \\
                                                                     &\leq  \frac{1}{(t-k+1)_2}  \sum_{j=0}^{\infty} \frac{1}{(t-k+2)^j} \\
                                                                     & < \frac{4}{3} \cdot   \frac{1}{(t-k+1)_2}.
\end{align*}
we get
\[
\frac{g_{t-k-1}-f_{t-k-1}}{(t-1)_{k+1}} \geq  \frac{g_t}{d_D^+(v_t)+1} \left( 1-\frac{4}{3} \sum_{j=1}^{k-1} \frac{1}{(t-j)_2}\right). 
\]
We are left to show the expression in part Part (d) is positive. We observe
\[
\sum_{j=1}^{k-1} \frac{1}{(t-j)_2} \leq \sum_{j=1}^{t-4} \frac{1}{(t-j)_2}=\frac{1}{(4)_2}+\cdots+\frac{1}{(t-1)_2}=\frac{1}{3}-\frac{1}{t-1} < \frac{1}{3}.
\]
here we used the assumption  $t \geq 5$.  We completed the inductive step for Part (d). 

An additional argument is needed for $k=t-s$ since $(v_t,v_s) \in E(D^+)$ and $(v_t,v_s) \not \in E(D)$.   We observe $s \geq 3$ since otherwise we do not need this argument.  We have
\[
g_{s}=\frac{g_{s-1}}{s-1}+\sum_{1 \leq j \leq s-t-1} \frac{g_{t-j}}{t-j}+\frac{g_t}{d_D^+(v_t)+1}+ \sum_{\underset{v_j \to v_{s}}{ j \geq t+1}} \frac{g_j}{d_D^+(v_j)},
\]
while
\[
f_{s}=\frac{f_{s-1}}{s-1}+\sum_{1 \leq j \leq s-t-1} \frac{f_{t-j}}{t-j}+\sum_{\underset{v_j \to v_{s}}{ j \geq t+1}} \frac{f_j}{d_D^+(v_j)}.
\]

 As we did previously in the the inductive proof, we have
 \[
\frac{g_{s-1}-f_{s-1}}{(t-1)_{t-s+1}} \geq  \frac{g_t}{d^+(w_t)+1} \left( 1-\frac{4}{3} \sum_{j=1}^{t-s-2} \frac{1}{(t-j)_2} - \sum_{j=1}^{t-s-1} \frac{1}{(t-j)_{t-s-j+1}}-\frac{1}{(t-1)_{t-s}}  \right).
\]
We need only to prove the first inequality of Part (d) for $k=t-s$. If $t-s=3$, then we prove Part (d) for $k=3$ directly.  For $t-s \geq 4$, we have
\begin{align*}
\sum_{j=1}^{t-s-1} \frac{1}{(t-j)_{t-s-j+1}}+\frac{1}{(t-1)_{t-s}} &< \frac{1}{(s+1)_2} \left( \sum_{j=0}^{\infty} \frac{1}{(s+2)^j}+\frac{1}{(t-1)_{t-s-2}} \right)\\
                                                                                                              & < \frac{1}{(s+1)_2} \left( \frac{5}{4}+ \frac{1}{(s+2)(s+3)}  \right)\\
                                                                                                              & <  \frac{4}{3} \cdot \frac{1}{(s+1)_{2}}
\end{align*}
We used facts $s \geq 3$ and $t-s \geq 4$ to prove the inequalities above.   For the range of $ t-s+1 \leq k \leq t-2$, this can be proved along the same lines as the range of $3 \leq k \leq t-s$.
  \end{proof}

  Using Proposition $\ref{addedge0}$, we can now compare $\gamma(D)$ and $\gamma(D^+)$.

\begin{proposition} \label{addedge1}
For each  $D \in \D_n$, let $D^+$ be defined as in Definition \ref{DPlus}. Then $\gamma(D^+)> \gamma(D)$.
\end{proposition}
\begin{proof}   Since the Perron vector has positive entries,  rescalling it by a positive number will not change the principal ratio.  Thus  we are able to assume $\psi$ satisfies
\[
\phi(v_2)= \psi(v_2).
\]
To prove the claim, it is enough to show $\phi(v_n) > \psi(v_n)$.  Suppose  not, i.e.,  $\phi(v_n) \leq  \psi(v_n)$.

 Recall Proposition \ref{addedge0}. If $t=3$   then we have $g_2 > f_2$ as Part (a), Proposition \ref{addedge0}.  For $t=4$,   we have $g_2 > f_2$  as  Part (b), Proposition \ref{addedge0}.  Since we assumed $\phi(v_n) \leq \psi(v_n)$,  we have $\psi(v_2)=g_2 \cdot \psi(v_n)> \phi(v_2)=f_2 \cdot \phi(v_n)$, which is a contradiction.  If $t \geq 5$,  then we apply  Part (d) of Proposition \ref{addedge0} with $k=t-2$ and get $g_2 > f_2$.
In the case of  $t=5$, we still have the same inequality. Therefore, we can find the same contradiction as the case of $t=4$.
\end{proof}

\section{Deleting edges to increase the principal ratio} \label{sec:deleteEdge}

We now consider another family of graphs $\D_n'$, disjoint from $\D_n$, which satisfy the properties necessary for extremality in Section $\ref{extremalStructure}$.


\begin{definition}\label{def:DnPrime}
For each  $n$,  let $\D'_n$ be a family of directed graphs where each $D \in \D_n'$ on vertex set $\{v_1,\dots,v_n\}$ satisfies the following properties:
\begin{description}
\item [(i)] The shortest path from $v_1$ to $v_n$ is of length $n-1$ and is denoted by $v_1,v_2,\ldots,v_n$.
\item[(ii)] For each $2 \leq i \leq n-1$, $d^+(v_i)=i$.
\item[(iii)] $v_2 \in V_{\max}$, $v_n \in V_{\min}$, and $\dist( V_{\max}, V_{\min})=\dist(v_2,v_n)=n-2$.
\item[(iv)] $d^+(v_n) \geq 2$.
\item[(v)] $N^+(v_n) \not = \{v_1,v_2\}$.
\end{description}
\end{definition}

For each $D\in \D_n'$, we now define an associated graph $D^-$ identical to $D$ except for the deletion of a single edge.

\begin{definition} \label{DMinus}
For each $D \in D_n'$, let $3 \leq t \leq n-1$ be the largest integer such that $(v_n,v_t) \in E(D)$. We define $D^-$ as the directed graph whose edge set is $E(D) \setminus \{(v_n,v_t)\}$, as illustrated in Figure $\ref{fig:DMinus}$.
\end{definition}

\begin{figure}[t]
\centering
\includegraphics[scale=.5]{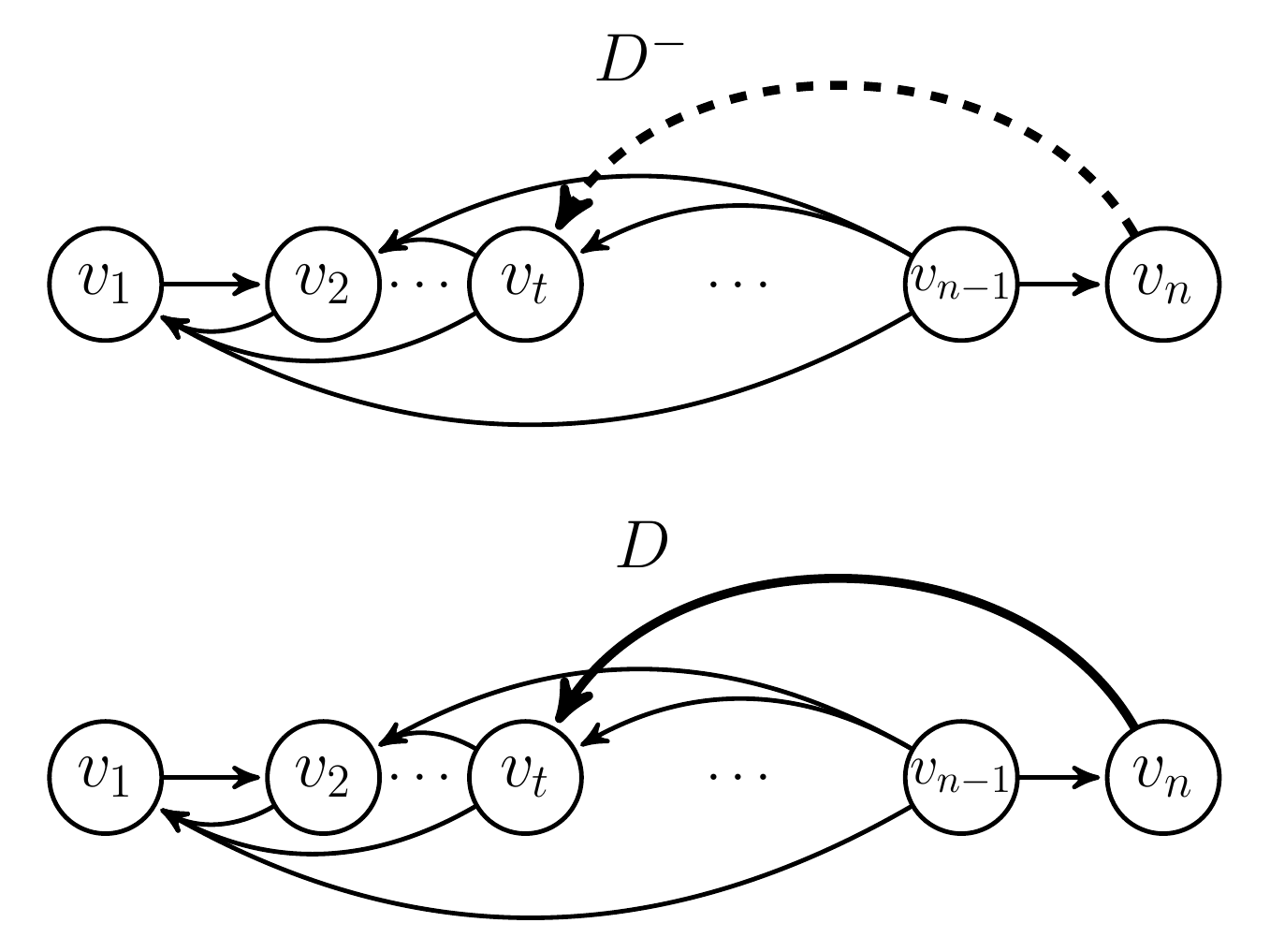}
\caption{$D$ and $D^-$. A dashed edge indicates the absence of that edge.}\label{fig:DMinus}
\end{figure}

Analogous to how Proposition \ref{addedge0} allowed us to compare the principal ratios of $D$ and $D^+$, the following proposition will allow us to compare the principal ratios of $D$ and $D^-$.

\begin{proposition} \label{deleteedge0}
For each $D \in \D_n'$, let $D^-$ be defined as in Definition \ref{DMinus}. Assume $\phi$ and $\psi$ are the Perron vectors of the  probability transition matrices of $D$ and $D^-$ respectively.  Moreover, suppose $\phi(v_i)=f_i \cdot \phi(v_n)$ and  $\psi(v_i)=g_i \cdot \psi(v_n)$ for each $1 \leq i \leq n$. We have
\begin{enumerate}
\item[(a)]    $f_i=g_i$ for $t \leq i \leq n$ .
\item[(b)] $\tfrac{g_{t-1}-f_{t-1}}{t-1}=\tfrac{1}{d_D^+(v_n)}$.
\item[(c)] $0< \tfrac{1}{d_D^+(v_n)} \left( 1-\tfrac{1}{(t-1)(d_D^+(v_n)-1)} \right) \leq \tfrac{g_{t-2}-f_{t-2}}{(t-1)_2} \leq \tfrac{1}{d^+_D(v_n)}$.
\end{enumerate}
If $t \geq 5$, then additionally we have
\begin{enumerate}
\item[(d)]  For $3 \leq k \leq t-2$, we have \\
$\tfrac{g_{t-k}-f_{t-k}}{(t-1)_k} \geq \tfrac{1}{d_D^+(v_n)} \left(1-\tfrac{4}{3} \sum_{j=1}^{k-2} \tfrac{1}{(t-j)_2}  -\tfrac{1}{d_D^+(v_n)-1} \sum_{j=1}^{k-1} \tfrac{1}{(t-1)_j} \right)>0 .$
\item[(e)]  For $3 \leq k \leq t-2$, we have     $\tfrac{g_{t-k}-f_{t-k}}{(t-1)_k} \leq \tfrac{1}{d_D^+(v_n)}.$
\end{enumerate}
\end{proposition}
\begin{proof}
We observe $d_D^+(v_i)=d_{D^-}^+(v_i)=i$ for each $1 \leq i \leq n-1$ and $d_D^+(v_n)-1=d_{D^-}^+(v_n)$.  Also, $f_n=g_n=1$.  Part (a) is a simple consequence of  Proposition \ref{equalpart}.   We can verify Part (b)  and Part (c) directly.  We note when we check Part (c),  there are two cases depending on whether $(v_n,v_{t-1})$ is an edge or not.  
 If $t \in \{3,4\}$, then we do not need Part (d) or Part (e).  Thus we assume $t \geq 5$.   We will  prove Part (d) and Part (e) simultaneously using induction.

  The base case is $k=3$. We have two cases. \\
  
  \noindent \underline{Case 1}: $(v_n,v_{t-3}) \in E(D)$. \\
  
  
Using the equation $(\ref{recursion})$, we have
\begin{equation} \label{delete1}
g_{t-2}=\frac{g_{t-3}}{t-3}+\frac{g_{t-1}}{t-1}+\frac{1}{d_D^+(v_n)-1}+\sum_{\underset{v_j \to v_{t-3}}{t \leq j \leq n-1}} \frac{g_j}{d_D^+(v_n)}
\end{equation}
Similarly,
\begin{equation} \label{delete2}
f_{t-2}=\frac{f_{t-3}}{t-3}+\frac{f_{t-1}}{t-1}+\frac{1}{d_D^+(v_n)}+\sum_{\underset{v_j \to v_{t-3}}{t \leq j \leq n-1}} \frac{f_j}{d_D^+(v_n)}
\end{equation}
Subtracting $f_{t-2}$ from $g_{t-2}$, rearranging terms followed by dividing both sides by $(t-1)_2$, we have
\begin{equation} \label{delete3}
\frac{g_{t-3}-f_{t-3}}{(t-1)_3}=\frac{g_{t-2}-f_{t-2}}{(t-1)_2}-\frac{g_{t-1}-f_{t-1}}{(t-1)(t-1)_2}-\frac{1}{d_D^+(v_n)(d_D^+(v_n)-1)(t-1)_2}
\end{equation}
Applying Part (a) to Part (c), we have
\[
\frac{g_{t-3}-f_{t-3}}{(t-1)_3} \geq \frac{1}{d_D^+(v_n)} \left(1-\frac{1}{(t-1)_2}-\frac{1}{d_D^+(v_n)-1} \left( \frac{1}{t-1}+\frac{1}{(t-1)_2} \right)   \right) >0
\]
The above quantity is clearly positive since $t \geq 5$.  Therefore, we obtained the base case for Part (d).  From  \eqref{delete3}, if we apply Part (b) and Part (c) as well as  $t \geq 5$, then we get $\tfrac{g_{t-3}-f_{t-3}}{(t-1)_3} \leq \tfrac{1}{d^+(v_n)}$, which is the base case for Part (e). \\

  \noindent \underline{Case 2}: $(v_n,v_{t-3}) \not\in E(D)$. \\

 If $(v_n,v_{t-3})$ is not an edge, then $\tfrac{1}{d_D^+(v_n)-1}$ is missing from \eqref{delete1} and $\tfrac{1}{d_D^+(v_n)}$ is missing in \eqref{delete2}. However, \eqref{delete3} still holds in this case.  We can prove the base case for Part (d) and Part (e) similarly.

For the inductive step, we assume Part (d) and Part (e) are true for all $3 \leq i \leq k$. We first deal with the case where $(v_n,v_{t-k})$ is an edge.  Again, from  equation $(\ref{recursion})$ we have
\begin{equation} \label{delete4}
g_{t-k}=\frac{g_{t-k-1}}{t-k-1}+\sum_{1 \leq j  \leq k-1} \frac{g_{t-j}}{t-j}+\frac{1}{d_D^+(v_n)-1}+\sum_{t \leq j \leq n-1} \frac{g_j}{d_D^+(v_j)}.
\end{equation}
Similarly, for $f_{t-k}$, we have
\begin{equation} \label{delete5}
f_{t-k}=\frac{f_{t-k-1}}{t-k-1}+\sum_{1 \leq j  \leq k-1} \frac{f_{t-j}}{t-j}+\frac{1}{d_D^+(v_n)-1}+\sum_{t \leq j \leq n-1} \frac{f_j}{d_D^+(v_j)}.
\end{equation}
We  solve for $\tfrac{g_{t-k-1}-f_{t-k-1}}{t-k-1}$ and then divide both sides of the  equation by $(t-1)_k$. We get
\begin{equation} \label{delete6}
\frac{g_{t-k-1}-f_{t-k-1}}{(t-1)_{k+1}}=\frac{g_{t-k}-f_{t-k}}{(t-1)_{k}}-\sum_{j=1}^{k-1}\frac{g_{t-j}-f_{t-j}}{(t-j)(t-1)_{k}}-\frac{1}{d_D^+(v_n)(d_D^+(v_n)-1)(t-1)_k}.
\end{equation}
 By the  inductive hypothesis for Part (d) and Part (e), we get $\tfrac{g_{t-k-1}-f_{t-k-1}}{(t-1)_{k+1}} \leq \tfrac{g_{t-k}-f_{t-k}}{(t-1)_{k}} \leq \tfrac{1}{d^+(v_n)}$, which  proves the inductive step for Part (e).
  
From  the inductive hypothesis of Part (d), we get  
  \begin{equation} \label{delete7}
  \frac{g_{t-k}-f_{t-k}}{(t-1)_k} \geq \frac{1}{d_D^+(v_n)} \left(1-\frac{4}{3} \sum_{j=1}^{k-2} \frac{1}{(t-j)_2}  -\frac{1}{d_D^+(v_n)-1} \sum_{j=1}^{k-1} \frac{1}{(t-1)_j} \right)
  \end{equation}
  From the inductive hypothesis for Part (e), we have 
  \begin{equation} \label{delete8}
  \sum_{j=1}^{k-1}\frac{g_{t-j}-f_{t-j}}{(t-j)(t-1)_{k}} =     \sum_{j=1}^{k-1} \frac{g_{t-j}-f_{t-j}}{(t-1)_j} \cdot \frac{1}{(t-j)_{k-j+1}}  \leq \frac{1}{d^+(v_n)}  \sum_{j=1}^{k-1} \frac{1}{(t-j)_{k-j+1}}
  \end{equation}
  Putting \eqref{delete6}, \eqref{delete7} and \eqref{delete8} together, we get 
\[
\frac{g_{t-k-1}-f_{t-k-1}}{(t-1)_{k+1}} \geq \frac{1}{d_D^+(v_n)} \left(1-\frac{4}{3} \sum_{j=1}^{k-2} \frac{1}{(t-j)_2}-\sum_{j=1}^{k-1} \frac{1}{(t-j)_{k-j+1}}-\frac{1}{d_D^+(v_n)-1}\sum_{j=1}^k \frac{1}{(t-1)_j}  \right).
\]
By the same lines as the proof of Proposition \ref{addedge0}, we can show $\sum_{j=1}^{k-1} \tfrac{1}{(t-j)_{k-j+1}} < \tfrac{4}{3} \cdot \tfrac{1}{(t-k+1)_2}$. Therefore, we proved 
\[
\frac{g_{t-k}-f_{t-k}}{(t-1)_k} \geq \frac{1}{d_D^+(v_n)} \left(1-\frac{4}{3} \sum_{j=1}^{k-2} \frac{1}{(t-j)_2}  -\frac{1}{d_D^+(v_n)-1} \sum_{j=1}^{k-1} \frac{1}{(t-1)_j} \right)
\]
 We note 
 \[
\frac{4}{3} \sum_{j=1}^{k-2} \frac{1}{(t-j)_2}+\frac{1}{d^+(v_n)-1} \sum_{j=1}^{k-1} \frac{1}{(t-1)_j}  \leq \frac{4}{3} \left( \frac{1}{3}-\frac{1}{t-1} \right)+ \frac{1}{t-1} \sum_{i=0}^{\infty} \frac{1}{2^i} < \frac{4}{9}+\frac{3}{8}<1
\]
Here we applied the assumption $t \geq 5$. Thus, that the expression in Part (d) is positive follows from the inequality above. We established  the inductive step of Part (d) in the case where $(v_n,v_k)$ is an edge.  For the case where  $(v_n,v_{t-k})$ is not an edge, we note $\frac{1}{d^+(v_n)-1}$ is missing from \eqref{delete4} and $\frac{1}{d^+(v_n)}$ is missing from \eqref{delete5}. The argument goes along the same lines.
\end{proof}

Using Proposition $\ref{deleteedge0}$, we can now compare $\gamma(D)$ and $\gamma(D^-)$.

\begin{proposition} \label{deleteedge1}
For each $D \in \D_n'$, let $D^-$ be defined as Definition \ref{DMinus}. We have $\gamma(D^-)> \gamma(D)$.

\end{proposition}
\begin{proof} We use the same idea as the proof for Proposition \ref{addedge1}.  We rescale $\psi$ such that $\phi(v_2)=\psi(v_2)$ and show $\psi(v_n)< \phi(v_n)$.  Suppose  $\psi(v_n) \geq  \phi(v_n)$.
 We will show $g_2 > f_2$ which will yield $\psi(v_2)> \phi(v_2)$ since $\psi(v_2)=g_2 \cdot \psi(v_n)$ and  $\phi(v_2)=f_2 \cdot \phi(v_n)$ as well as the assumption $\psi(v_n) \geq \phi(v_n)$. If $t \in \{3,4\}$, then $g_2 > f_2$ follows either from Part (b) or Part (c) of Proposition \ref{deleteedge0}.  If $t \geq 5$, then we will apply Part (d) of Proposition \ref{deleteedge0} with $k=t-2$  to get  $g_2 > f_2$.  We  draw the contradiction similarly.
\end{proof}

\section{Proof of Theorem $\ref{thm1}$} \label{sec:mainThmProof}

We can now prove Theorem $\ref{thm1}$ as a consequence of Propositions $\ref{keylemma1}-\ref{deleteedge1}$. \\

\noindent {\bf Proof of Theorem \ref{thm1}:} We will show that the extremal graphs achieving the maximum of the principal ratio over all strongly connected $n$-vertex graphs are precisely $D_1,D_2$, and $D_3$ and that their principal ratio is indeed as claimed in Theorem \ref{thm1}.

We will use the fact that $D_1$ has principal ratio as follows, which we will prove at the end of this section:
\begin{align*}
\gamma(D_1)&=\frac{2}{3}\left(\frac{n}{n-1}+\frac{1}{(n-1)!}\sum_{i=1}^{n-3}i!\right)(n-1)!
\end{align*}

Assume $D$ is extremal, i.e. its principal ratio is at least as large as that of any directed graph on $n$ vertices. For any (strongly connected) directed graph $D$, we have $\dist(V_{\max},V_{\min}) \leq n-2$ by Proposition \ref{reducedist}. If $D$ is such that $\dist(V_{\max},V_{\min}) \leq n-3$, then $D$ is not extremal since by Proposition \ref{smalldist}, we have  $\gamma(D)< \gamma(D_1)$. So $\dist(V_{\max},V_{\min})=n-2$, where $v_2 \in V_{\max}$, $v_n \in V_{\min}$, and $v_2,v_3,\ldots,v_n$ is a shortest path from $v_2$ to $v_n$.  If $D$ is extremal, then $\gamma(D)\geq \gamma(D_1)>\tfrac{2}{3}(n-1)!$. So, applying Proposition \ref{special} and Proposition \ref{degree}, we can assume further that $v_1,v_2,\ldots,v_n$ is a shortest path from $v_1$ to $v_n$, $d^+(v_i)=i $ for $i \in \{2,3\}$, and $d^+(v_i) \geq \lfloor  \tfrac{2i}{3} \rfloor$ for $4 \leq i \leq n$.

Now, if $D \in \D_n$, then $D$ is not extremal by Proposition 10. Similarly, if $D \in \D_n'$, $D$ is not extremal by Proposition 12.

So, $D\not \in \D_n$ and $D \not\in \D_n'$. Since $D \not  \in \D_n$ but satisfies all properties for inclusion in $\D_n$ except $(v)$ in Definition \ref{def:Dn}, it must be that $d^+(v_i)=i$ for each $2\leq i \leq n-1$. Then, since $D\not \in \D_n$ but satisfies all properties for inclusion in $\D_n$ except either (iv) or (v) in Definition \ref{def:DnPrime}, either $d^+(v_n)=1$ or $N^+(v_n)=\{v_1,v_2\}$. In the former case, if $N^+(v_n)=\{v_j \}$ for $j\geq 3$, then arguing along the same lines as in the proof of Proposition \ref{deleteedge1}, one has $\gamma(D)<\gamma(D_1)$; otherwise $D=D_1$ or $D=D_2$. In the latter case, $D=D_3$.

Lastly, we show that $D_1,D_2$, and $D_3$ all have the same principal ratio. Assume $\phi, \psi, \tau$ are the Perron vectors of $D_1, D_2$, and $D_3$ respectively. Scale their Perron vectors  so that all three agree on the $n$th coordinate.  By Proposition \ref{writeFG},  we know there exist (positive) functions $f_i, g_i ,h_i$ so that
\begin{align*}
\phi(v_i) &= f_i \cdot \phi(v_n) \\
\psi(v_i) &= g_i \cdot \psi(v_n) \\
\tau(v_i) &= h_i \cdot \tau(v_n)
\end{align*}
By Proposition \ref{equalpart},  we note $f_i=g_i=h_i$ for $2 \leq i \leq n$.   We can prove the following inequalities for $f_i$. 
\begin{enumerate}
\item[(a)] $\tfrac{f_{n-1}}{n-1}=\tfrac{f_{n-2}}{(n-1)_2}=f_n$
\item[(b)] For each $3 \leq k \leq t-2$, we have $f_n \left( 1-\tfrac{4}{3} \sum_{j=1}^{k-2} \tfrac{1}{(n-j)_2} \right) \leq \tfrac{f_{n-k}}{(n-1)_k} \leq  f_n.$
\end{enumerate}
The proof of Part (a) and Part (b) uses the same argument as the proof of Proposition \ref{addedge0} and it is omitted here.  If $n \leq 5$, then we can verify  $\max \{f_i: 1 \leq i \leq n\}=f_2$ and  $\min \{f_i: 1 \leq i \leq n\}=f_n$ directly.  
Suppose $n \geq 6$.   By  Part (b), for each $3 \leq k \leq n-2$ we have
\begin{align*}
\frac{f_{n-k}-f_{n-k+1}}{(n-1)_k}&=\frac{f_{n-k}}{(n-1)_k}-\frac{f_{n-k+1}}{(n-1)_{k-1}(n-k)}\\
                                                          & \geq \frac{f_{n-k}}{(n-1)_k}-\frac{f_n}{n-k} \\
                                                          & \geq   f_n\left(1-\frac{4}{3} \sum_{j=1}^{k-2} \frac{1}{(n-j)_2} -\frac{1}{n-k}\right)\\
                                                          &\geq f_n\left(1-\frac{4}{3(n-k+2)}-\frac{1}{n-k}\right)\\
                                                          & > f_n\left( 1-\frac{1}{3}-\frac{1}{2} \right)=\frac{f_n}{6} 
\end{align*}
 We can check $f_1 > f_n$ easily. Therefore, we obtain  $\max \{f_i: 1 \leq i \leq n\}=f_2$ and  $\min \{f_i: 1 \leq i \leq n\}=f_n$ .  The same holds for $g_2$ and $h_2$, which completes the proof.\hfill $\square$ \\

We now compute the stationary distribution and principal ratio of $D_1$, completing the proof of Theorem \ref{thm1}. \\





\noindent {\bf Claim A}.
{\em Let $D_1$ be as defined in the statement of Theorem \ref{thm1}, and let $\phi$ be the Perron vector associated with the transition probability matrix $P$ of $D_1$. Then
\[
\gamma(D_1)=\frac{2}{3}\left(\frac{n}{n-1}+\frac{1}{(n-1)!}\sum_{i=1}^{n-3}i!\right)(n-1)!.
 \]
where
\[
\min_{1 \leq i \leq n} \phi(v_i)= \phi({v_n}) \textrm{ and }  \max_{1 \leq i \leq n} \phi(v_i)=\phi({v_2}).
\]
}

%
%
%
%
%
%
%
%

\noindent
{\bf Proof}: Since we are concerned with the ratio of the maximum entry and the minimum entry of the Perron vector, rescaling the Perron vector  by a positive number will not affect our result. We assume $x=(x_1,x_2,\dots, x_n)$ with $x_n=1$ such that $x P=x$, where
\[
P=
\begin{pmatrix}
0 & 1 & 0 & \cdots  & 0 & 0 \\
1/2 & 0 & 1/2 & \cdots & 0 & 0 \\
1/3 & 1/3 & 0 & \cdots & 0 & 0 \\
1/4 & 1/4 & 1/4 & \cdots & 0 & 0 \\
\vdots & \vdots & \vdots & \ddots & \vdots & \vdots \\
\tfrac{1}{n-1} & \tfrac{1}{n-1} & \tfrac{1}{n-1} & \cdots & 0 & \tfrac{1}{n-1}\\
1 & 0 & 0 & \cdots & 0 & 0 \\
\end{pmatrix}.
 \]
Suppose $P=(p_1,p_2,\ldots,p_n)$ where $p_i$ is the $i$-th column of $P$ for each $1 \leq i \leq n$. From $x_1=x \cdot p_1$ and $x_2=x \cdot p_2$, we have
\begin{equation} \label{ex:1}
x_2=\frac{4}{3}x_1-\frac{2}{3},
\end{equation}
where we used the assumption $x_n=1$. As  $x_3=x \cdot p_3$ and $x_1= x \cdot p_1$, we have
\begin{equation} \label{ex:2}
x_3=\frac{3}{4}x_1-\frac{3}{4}.
\end{equation}
For each $2 \leq k \leq n-1$, we define
\begin{align*}
a_k &= \frac{2k}{(k+1)(k-1)!} \\
b_k &= \frac{k}{(k+1)(k-1)!}\sum_{i=0}^{k-2}i!
\end{align*}
For each $2 \leq k \leq n-1$, we will show
\begin{equation} \label{ex:3}
x_k =a_kx_1-b_k.
\end{equation}
We will prove \eqref{ex:3} by induction on $k$. The cases $k=2$ and $k=3$ are given by \eqref{ex:1} and \eqref{ex:2} respectively.  Assume \eqref{ex:3} is true up to $l$ for some $3 \leq  l \leq n-2$. Using $x_{l+1}=x \cdot p_{l+1} $ and $x_{l-1}=x \cdot p_{l-1}$,  we have
\begin{align*}
x_{l+1} &=\frac{l+1}{l+2}\left(x_{l-1}-\frac{x_{l-2}}{l-2}\right) \\
&= \frac{l+1}{l+2}\left(\left(a_{l-1}x_1-b_{l-1}\right)-\frac{a_{l-2}x_1-b_{l-2}}{l-2}\right) \\
&=\frac{l+1}{l+2}\left(a_{l-1}-\frac{a_{l-2}}{l-2}\right)x_1 - \frac{l+1}{l+2}\left(b_{l-1}-\frac{b_{l-2}}{l-2}\right) \\
&=a_{l+1}x_1-b_{l+1}
\end{align*}
The inductive hypothesis  and an elementary computation gives:

\begin{align*}
a_{l+1} &= \frac{l+1}{l+2}\left(a_{l-1}-\frac{a_{l-2}}{l-2}\right) \\
&= \frac{l+1}{l+2}\left(\frac{2(l-1)}{l(l-2)!}-\frac{2(l-2)}{(l-1)!}\right) \\
&=\frac{2(l+1)}{(l+2)l!} \\
& \\
b_{l+1} &=\frac{l+1}{l+2}\left(b_{l-1}-\frac{b_{l-2}}{l-2}\right) \\
&= \frac{l+1}{(l+2)}\left(\left(\frac{(l-1)^2}{l!}\right)\sum_{i=0}^{l-3}i!-\left(\frac{l(l-2)}{l!}\right)\sum_{i=0}^{l-4}i!\right) \\
&=\frac{l+1}{(l+2)l!}\left(l(l-2)\left(\sum_{i=0}^{l-3}i!-\sum_{i=0}^{l-4}i!\right) +\sum_{i=0}^{l-3}i! \right) \\
&=\frac{l+1}{(l+2)l!}\left(l(l-2)!+\sum_{i=0}^{l-3}i!\right) \\
&=\frac{l+1}{(l+2)l!}\sum_{i=0}^{l-1}i!
\end{align*}
We have completed the proof of \eqref{ex:3}. Since $x_n= x \cdot p_n$, we have
\begin{equation} \label{ex:4}
x_n=\frac{x_{n-1}}{n-1}.
\end{equation}
Recall the assumption $x_n=1$. Using \eqref{ex:3} with $k=n-1$ and solving for $x_1$ in \eqref{ex:4}, we obtain
\[
x_1=\frac{n(n-2)!}{2}+\frac{1}{2} \sum_{i=0}^{n-3} i!.
\]
We already  have an explicit expression for entries of $x$. We claim
\[
x_2> x_1 > x_3 > x_4 > \dots > x_{n-1} > x_n
\]
We can verify $x_2 > x_1 >x_3$ and $x_{n-1} > x_n$ directly. To prove the remaining inequalities, for each $3 \leq k \leq n-2$, \eqref{ex:3} yields
\begin{align}
x_k &= \frac{2k}{(k+1)(k-1)!}x_1-\frac{k}{(k+1)(k-1)!} \sum_{i=0}^{k-2} i! \label{ex:5} \\
x_{k+1} &=\frac{2(k+1)}{(k+2)k!}x_1-\frac{(k+1)}{(k+2)k!} \sum_{i=0}^{k-1} i! \label{ex:6}
\end{align}
We first multiply \eqref{ex:6} by a factor $-\tfrac{k^2(k+2)}{(k+1)^2}$ and add the resulted equation to \eqref{ex:5}. We get the following equation
\[
x_k-\frac{k^2(k+2)}{(k+1)^2}x_{k+1}=\frac{k}{k+1}.
\]
The equation above implies $x_k \geq x_{k+1}$ for each $3 \leq k \leq n-2$. We have finished the proof of the claim. As  the Perron vector $\phi$ is a positive multiplier of $x$, we have
\[
\min_{1 \leq i \leq n} \phi(v_i)= \phi({v_n}) \textrm{ and }  \max_{1 \leq i \leq n} \phi(v_i)=\phi({v_2}).
\]

Finally, we are able to compute
\begin{align*}
\frac{\phi({v_{2}})}{\phi({v_n})}&= \frac{x_{2}}{x_n}\\
     &=\frac{2n(n-2)!}{3}+\frac{2}{3} \sum_{i=0}^{n-3} i! -\frac{2}{3}\\
  &=\frac{2}{3}\left(\frac{n}{n-1}+\frac{1}{(n-1)!}\sum_{i=1}^{n-3}i!\right)(n-1)!.
 \end{align*}
 
 This completes the proof of Theorem \ref{thm1}.
\hfill $\square$

\section{A sufficient condition for a tightly bounded principal ratio} \label{s3}

So far, we have shown that the maximum of the principal ratio over all strongly connected $n$-vertex directed graphs is $(2/3 + o(1))(n-1)!$. On the other hand, the minimum of the principal ratio is 1 and is achieved by regular directed graphs. In this section, we examine conditions under which the principal ratio is ``close" to the minimum of 1.

An important tool in our analysis will be notion of circulation defined by Chung \cite{chung2}. In a directed graph $D$, consider a function $F : E(D) \to {\mathbb R}^+ \cup \{0 \}  $ that assigns
to each directed edge $(u, v)$ a nonnegative value $F(u, v)$. $F$ is said to be a
{\it circulation} if at each vertex $v$, we have
\[
\sum_{u: u\in N^-(v)} F(u,v)=\sum_{w: w \in N^+(v)} F(v,w).
\]
For a circulation $F$ and a directed edge $e=(u,v)$, we will write $F(e)$ for $F(u,v)$ in some occasions. If $\phi$ is the Perron vector of the transition probability matrix $P$, then   Claim 1 in \cite{chung2} tells us that we can associate a circulation $F_{\phi}$ to $\phi$, where
\[
F_{\phi}(v,w)=\frac{\phi(v)}{d^+(v)}.
\]
We note that the circulation $F_{\phi}$ has the following property.  At each vertex $v$, we have
\begin{equation} \label{g:eq1}
\sum_{u: u\in N^-(v)} F_{\phi}(u,v)=\phi(v)=\sum_{w: w \in N^+(v)} F_{\phi}(v,w).
\end{equation}
We will  repeatedly  use \eqref{g:eq1} in the proof of the following theorem.

\begin{theorem} \label{sufficientCond}
 Let $D=(V,E)$ be a strongly connected  directed graph and $\phi$ be the Perron vector of the transition probability matrix $P$.  If there are positive constants $a,b,c,d,\epsilon$ such that

\begin{description}
 \item[(i)] $(a-\epsilon)n \leq d^+(v), d^-(v) \leq (a+\epsilon) n $ for all $v \in V(D)$ and
\item[(ii)] $|E(S,T)| \geq b |S| |T|$ for all disjoint subsets $S$ and $T$ with $|S|\geq cn$ and $|T|\geq dn$,
\end{description}

\noindent then we have
 \[
\gamma(D) \leq  \frac{1}{C}   \textrm{ for  }
 C=\frac{b(a-5\epsilon)(a-\epsilon)}{4(a+\epsilon)^2}.
 \]
\end{theorem}

Before proceeding with the proof of Theorem $\ref{sufficientCond}$, we illustrate that neither the degree condition (i), nor the discrepancy condition (ii) alone guarantee a small principal ratio. We first give a construction which satisfies the degree requirement but fails the discrepancy condition and has principal ratio linear in $n$.

\begin{example} \label{DegreeEx}
Construct a directed graph $D$ on $2n+1$ vertices as follows: take two copies of $D_n$, the complete directed graph on $n$ vertices, as well as an isolated vertex $b$. Add an edge from each vertex in the first copy of $D_n$ to $b$ and an edge from $b$ to each vertex in the second copy of $D_n$. Finally, select a distinguished vertex from the first copy of $D_n$, which we denote $e$, and a distinguished vertex from the second copy of $D_n$, which we denote $d$, and add edge $(d,e)$. Let $A$ denote the induced subgraph of the first copy of $D_n$ obtained by deleting vertex $e$; similarly, $C$ is the induced subgraph obtained by deleting vertex $d$ from the second copy of $D_n$. See Figure $\ref{fig:DegreeEx}$ for an illustration.
\begin{figure}[htbp]
\centering
\includegraphics[scale=.15]{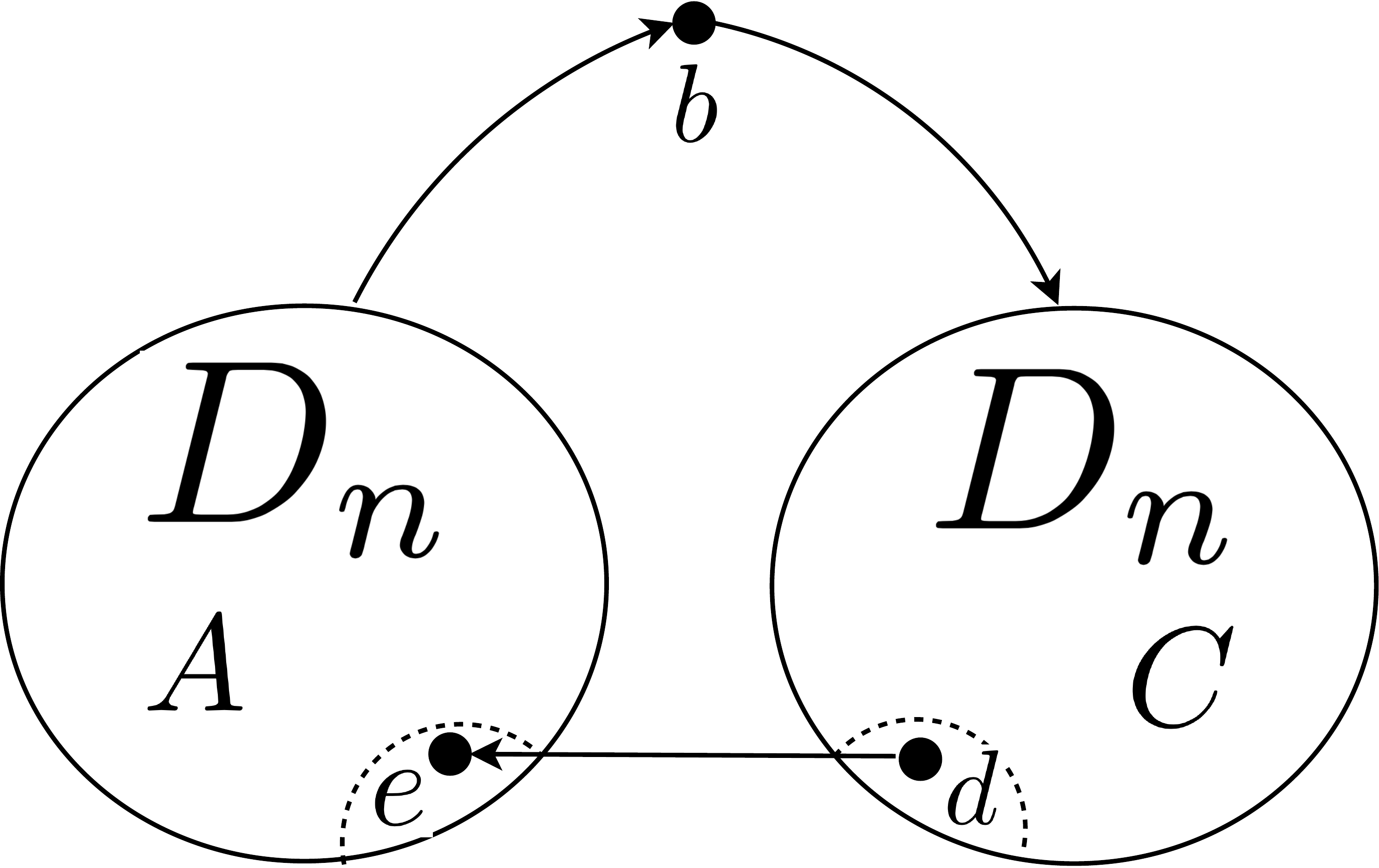}
\caption{The construction in Example $\ref{DegreeEx}$.}
\label{fig:DegreeEx}
\end{figure}
\end{example}

\begin{proposition} The construction $D$ in Example $\ref{DegreeEx}$ satisfies the degree condition of Theorem $\ref{sufficientCond}$ but not the discrepancy condition. The (unscaled) Perron vector of $D$ is given by

\[ \phi(u)=\begin{cases}
      1 & u \in V(A) \\
      \frac{n+1}{n} & u=b \\
      \frac{(n+1)^2(n-1)}{n^2} & u \in V(C) \\
      n+1 & u=d \\
      2 & u=e
   \end{cases}.
\]
Consequently, $\gamma(D)=\frac{\max_u \phi(u)}{\min_u \phi(u)}=n+1$. \end{proposition}

\noindent {\bf Proof}:
Observe that, for all $a \in V(A)$, $d_a^+=d_b^+=d_d^+=d_e^+=n$, and, for all $c \in V(C)$, $d_c^+=n-1$, thus $D$ satisfies the degree condition in Theorem $\ref{sufficientCond}$. However, $D$ fails the discrepancy condition since $E(V(A),V(C))=0$ where $|V(A)|=|V(C)|=n-1$. To compute the Perron vector of $D$, first observe that since $A$ and $C$ are vertex-transitive, $\phi(u)=\phi(a)$ for all $u,a \in V(A)$ and similarly $\phi(u)=\phi(c)$ for all $u,c \in V(C)$. Consider $a\in V(A)$. From $\phi=\phi P$, we obtain
\begin{align*}
\phi(a) &= \sum_{u \in N^-(a)} \phi(u) P(u,a) \\
&= \sum_{u \in N^-(a) \setminus V(A)} \phi(u)P(u,a)+ \sum_{u \in V(A)} \phi(u)P(u,a)\\
&= \frac{\phi(e)}{d_e^+}+\sum_{u \in V(A)} \frac{\phi(a)}{d_a^+}\\ &= \frac{\phi(e)}{n}+\frac{n-2}{n}\phi(a).
\end{align*}

In the same way as above, we also obtain equations for vertices $b,d,e$ and $c\in C$:
\begin{align*}
\phi(b) &= \frac{n-1}{n}\phi(a)+\frac{\phi(e)}{n} \\
\phi(c) &= \frac{\phi(b)}{n}+\frac{n-2}{n-1}\phi(c) + \frac{\phi(d)}{n} \\
\phi(d) &=\frac{\phi(b)}{n}+\phi(c) \\
\phi(e) &=\frac{n-1}{n}\phi(a)+\frac{\phi(d)}{n}
\end{align*}

We may set $\phi(a)=1$ and solve the above equations, yielding the result.
\hfill $\square$ \\

Next, we give a construction to illustrate the discrepancy condition alone is insufficient to guarantee a small principal ratio.

\begin{example} \label{DiscEx}
Construct a directed graph $D$ on $n+\sqrt{n}$ vertices as follows: First, construct the following graph from $\cite{chung2}$ on $\sqrt{n}$ vertices, which we denote $H_{\sqrt{n}}$. To construct $H_{\sqrt{n}}$, take the union of a directed cycle $C_{\sqrt{n}}$ consisting of edges $(v_j,v_{j+1})$ (where indices are taken modulo $\sqrt{n}$), and edges $(v_j,v_1)$ for $j=1,\dots,\sqrt{n}-1$.  Then, take a copy of $D_n$, the complete directed graph on $n$ vertices, and select from it a distinguished vertex $u$. Add edges $(v_1,u)$ and $(u,v_1)$. See Figure $\ref{fig:DiscEx}$ for an illustration.
\end{example}

\begin{figure}[htbp]
\centering
\includegraphics[scale=.4]{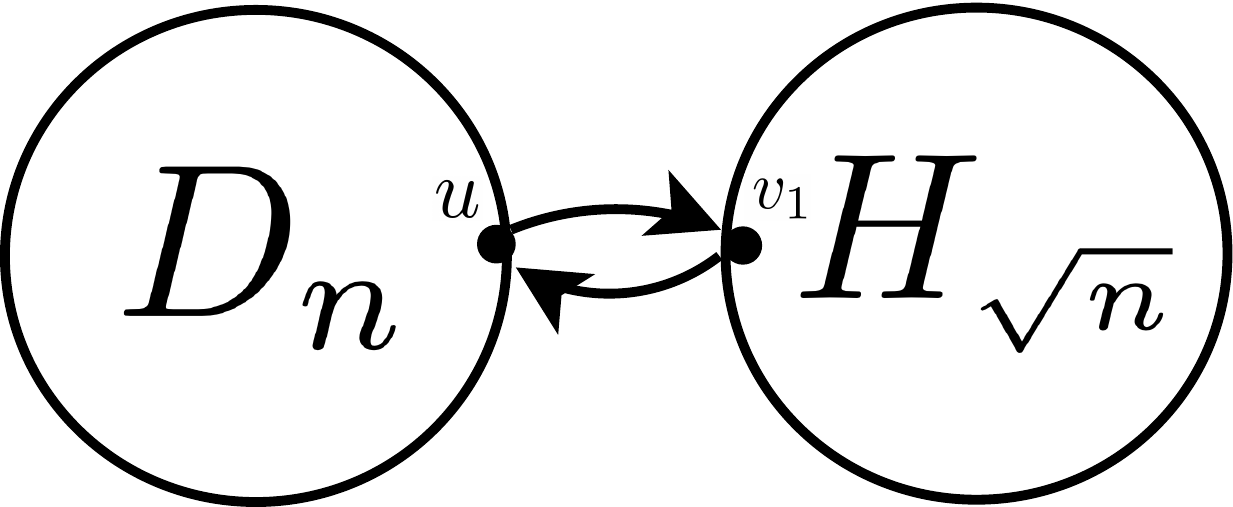}
\caption{The construction in Example $\ref{DiscEx}$.}
\label{fig:DiscEx}
\end{figure}
 It is easy to check $D$ as defined in Example \ref{DiscEx} satisfies the discrepancy condition in Theorem \ref{sufficientCond}, but not  the degree requirement (note $d^+_{v_{\sqrt{n}}}=1$ and $d^+_u=n$).  As noted in \cite{chung2}, the graph $H_{\sqrt{n}}$ has principal ratio $2^{\sqrt{n}-1}$.  Thus, $\gamma(D) \geq \gamma(H_{\sqrt{n}})=2^{\sqrt{n}-1}$.

Having shown that each condition in Theorem $\ref{sufficientCond}$ taken on its own is insufficient in ensuring a small principal ratio, we now prove that together they do provide a sufficient condition.\\

\noindent
{\bf Proof of Theorem $\ref{sufficientCond}$:}  We  assume
 \[
 \underset{v \in D(V)}{\max} \phi(v)=\phi(u) \textrm{ and  } \underset{v \in D(V)}{\min} \phi(v)=\phi(w).
 \]
We will show $\phi(w) \geq C \cdot \phi(u)$ instead, where $C$ is the constant in the statement of the theorem.
We  use  $U$  to denote the set $\{v \in N^-(u): \phi(v) \leq \tfrac{\phi(u)}{2}\}$. If $w \in N^-(u) \setminus U$, then we have nothing to show. Thus we assume $w \not \in N^-(u) \setminus U$. We consider the circulation $F_{\phi}$ associated with $\phi$ and recall \eqref{g:eq1}. By the definition of $U$, we have
\begin{align*}
\phi(u)=\sum_{v \in N^-(u)} F_{\phi}(v,u)&= \sum_{v \in U} F_{\phi}(v,u)+\sum_{v \in N^-(u) \setminus U} F(v,u)\\
                                  &  \leq  \sum_{v \in U} \frac{\phi(u)}{2(a-\epsilon)n}   +
                                      \sum_{v \in N^-(u) \setminus U} \frac{\phi(u)}{(a-\epsilon)n}\\
                                   & \leq  \frac{|U| \phi(u)}{2(a-\epsilon)n}+\frac{((a+\epsilon)n-|U| ) \phi(u)}{(a-\epsilon)n}
\end{align*}
Solving the inequality above, we have $|U| \leq 4\epsilon n$. Let $U'=N^-(u) \setminus U$. Then we have $|U'| \geq (a-5\epsilon) n$ as the assumption $|N^-(u)| \geq (a-\epsilon)n$.
If $|N^-(w) \cap U'| \geq \frac{|U'|}{2}$, then we have
\begin{align*}
\phi(w)&=\sum_{v \in N^-(w)} F_{\phi}(v,w) \\ &\geq \sum_{v \in N^-(w) \cap U'} F_{\phi}(v,w) \\
       & \geq \sum_{v \in N^-(w) \cap U' } \frac{\phi(u)}{2(a+\epsilon)n} \\
       &\geq \frac{(a-5\epsilon)\phi(u)}{4(a+\epsilon)}  \\
       & \geq C \cdot \phi(u)
\end{align*}
Therefore, we  assume $|N^-(w) \cap U'| <  \frac{|U'|}{2}$ in the remaining proof. We define $U''=U' \setminus N^-(w)$ and we have $|U''| \geq \tfrac{(a-5\epsilon)n}{2}$.
The assumption $|E(S,T)| \geq b |S||T|$ for any disjoint $S$ and $T$ implies
\begin{equation} \label{eq14}
|E(U'', N^-(w))| \geq b |U''| |N^-(w)| \geq \frac{b(a-5\epsilon)(a-\epsilon)n^2}{2}.
\end{equation}
Set  $\Phi_1=\sum_{v \in N^-(w)} \phi(v)$ and $E_1=E(U'',N^-(w))$. Using \eqref{g:eq1}, we will show the following inequality
\begin{equation} \label{eq15}
\Phi_1=\sum_{v \in N^-(w)} \sum_{z \in N^-(v)} F_{\phi}(z,v) \geq \sum_{e \in E_1} F_{\phi}(e) \geq \sum_{e \in E_1} \frac{\phi(u)}{2(a+\epsilon)n} \geq C(a+\epsilon)n \phi(u),
\end{equation}
\noindent where used inequality \eqref{eq14} in the last step.
By the definition of the circulation $F_{\phi}$, we have
\begin{equation} \label{eq16}
\phi(w)=\sum_{v \in N^-(w)} F_{\phi}(v,w) \geq \sum_{v \in N^-(w)} \frac{\phi(v)}{(a+\epsilon)n }=\frac{\Phi_1}{(a+\epsilon)n}.
\end{equation}
The combination of  inequalities \eqref{eq15} and \eqref{eq16} now completes the proof. 
\hfill $\square$

\end{document}